\documentclass{siamart0516}
\usepackage{amsmath}
\usepackage{amsfonts}
\usepackage{amssymb}
\usepackage{setspace}
\newsiamremark{rem}{Remark}
\title{Existence of the optimum for Shallow Lake type \textsuperscript{ }models
\footnotetext{\textbf{F}\MakeLowercase{\textbf{unding: }This research did not receive any specific grant from funding agencies in the public, commercial, or not-for-profit sectors.}}}
\author{Francesco Bartaloni}
\headers{Existence of the optimum for Shallow Lake type models}{Francesco Bartaloni}

\begin{document}

\maketitle
\vspace{-2mm}
\begin{center}
\small{Dipartimento di Matematica, Universit\`a di Pisa,\\
Largo B. Pontecorvo 5, 56127 Pisa, Italy.\\
E-mail address: \email{bartaloni@dm.unipi.it}}
\end{center}

\bigskip
\begin{abstract}
We consider the optimal control problem associated with a general
version of the well known shallow lake model, and we prove the existence
of an optimum in the class $L_{loc}^{1}\left(0,+\infty\right)$. Any
direct proof seems to be missing in the literature. Dealing with admissible
controls that can be unbounded (even locally) is necessary in order
to represent properly the concrete optimization problem; on the other
hand, the non-compactness of the control space together with the infinite
horizon setting prevents from having good \emph{a priori} estimates
- and this makes the existence problem considerably harder. We present
an original method which is in a way opposite to the classical control
theoretic approach used to solve finite horizon Mayer or Bolza problems.
Synthetically, our method is based on the following scheme: i) two
uniform localization lemmas providing, given $T\geq1$ and a maximizing
sequence of controls, another sequence of controls which is bounded
in $L^{\infty}\left(\left[0,T\right]\right)$ and still maximizing.
ii) A special diagonal procedure dealing with sequences which are
not extracted one from the other. iii) A \textquotedblleft{}standard\textquotedblright{}
diagonal procedure. The optimum results to be locally bounded by construction.\\
%\textbf{Keywords:} Optimal control, infinite horizon, non compact
%control space, uniform localization, convex-concave dynamics, logarithmic
%utility.
\end{abstract}

\begin{keywords}
Control, global optimization, non compact control space, uniform localization, convex-concave dynamics.
\end{keywords}

\onehalfspacing

\section{Introduction}

In this work we examine the optimal control problem related to a general
version of the Shallow Lake model, and we prove the existence of an
optimum. In the last fifteen years, a literature about this model
has grown up, but, in our knowledge, no direct existence proof has
been provided up to now. The optimal control problem has been introduced
in \cite{Maler}, and has been studied mostly via dynamic programming
(\cite{Kossioris}), or from the dynamical systems viewpoint (see
e.g. \cite{Kiseleva-Wagener 1}, \cite{Kiseleva-Wagener 2} and \cite{Maler}).
The latter approach consists in the analysis of the adjoint system
that is obtained coupling the state equation with the adjoint equation
given by the Pontryagin Maximum Principle. As it is well known, such
principle provides conditions for optimality that in general are merely
necessary.

The main technical difficulties in order to prove the existence of
an optimum arise from the fact that good \emph{a priori} estimates
for the controls and for the states are missing, because of the infinite
horizon setting and the unboundedness assumption on the set of admissible
controls. Indeed, the intimate nature of the model requires that one
may be allowed to choose a (locally integrable) control function that
reaches arbitrarily large values in a finite time. Also controls that
are arbitrarily near $0$ are allowed, and this produces similar effects
when the functional has logarithmic dependence on the control. In
this context the application of any compactness result is not straightforward.\\

Here we propose an original approach to the existence problem. In
a sense, we proceed the opposite direction respect to what is done
in the proof of some classical existence results for finite horizon
problems such as Filippov-Cesari theorem. In the latter kind of proof,
thanks to some \emph{a priori} estimates, Ascoli-Arzel\`a theorem is
applied in order to obtain an optimizing sequence of \emph{states}
converging to a candidate optimal state, which is proven to be almost
everywhere differentiable; then some convexity assumption on the dynamics
of the state equation allows to pointwise identify a control satisfying
the instance of the state equation involving the candidate optimal
state. Finally, such control is proven to be admissible by a measurable
selection argument. This is what is essentially needed in the case
of finite horizon Mayer problems; in the case of Bolza problems with
coercive dependence of the integral functional on the control, the
same scheme is, roughly speaking, applied to the couple $x_{n},\, J_{int}\left(x_{n},u_{n}\right)$,
where $\left(x_{n}\right)_{n}$ is an optimizing sequence of states
and $J_{int}\left(\mathrm{x},\mathrm{u}\right)\left(t\right)=\int_{T_{0}}^{t}L\left(s,\mathrm{x}\left(s\right),\mathrm{u}\left(s\right)\right)\mbox{d}s$
is the integral part of the objective functional; in this case, after
proving that the limit $x_{*}$ of $x_{n}$ has an admissible companion
control $u_{*}$, one also has to prove that $u_{*}$ is in the proper
relation with the limit of $J_{int}\left(x_{n},u_{n}\right)$. For
the details of the latter (complex) proof, see \cite{Fleming Rishel},
Chapter III, \S\ 5.

In other words, the classical control theoretic approach to the existence
problem starts with the convergence of the states and associated functionals
to some limit, and ends up with with a control function giving those
two limits the desired form; in particular no direct semi-continuity
argument for the functional is used.\\

In our approach, dealing with a functional of the type
\[
J\left(\mathrm{u}\right)=\int_{0}^{+\infty}e^{-\rho t}\left(\log\mathrm{u}\left(t\right)-c\mathrm{x}^{2}\left(t\right)\right)\mbox{d}t,
\]
we consider an optimizing sequence of locally integrable \emph{controls}
$\left(u_{n}\right)_{n}$ and, in order to bypass the absence of \emph{a
priori} estimates, we prove two\emph{ }uniform localization lemmas
(``from above'' and ``from below''). This way, for a fixed compact
interval $\left[0,T\right]$, we are able to find a sequence $\left(u_{n}^{T}\right)_{n}$
which is still optimizing and also uniformly bounded in $\left[0,T\right]$,
by two quantities $N\left(T\right)$, $\eta\left(T\right)$. By weak
(relative) compactness we can extract a sequence $\left(\bar{u}_{n}^{T}\right)_{n}$,
weakly converging in $L^{1}\left(\left[0,T\right]\right)$. We repeat
the process for bigger and bigger intervals, each time starting from
the maximizing sequence we ended up with in the previous step.

In order to merge properly the local (weak) limits, the standard diagonal
argument does not work, since we are in presence of two families of
sequences which \emph{a priori} are not extracted one from the other:
the ``barred'' converging sequences and the ``unbarred'' sequences
obtained by applying the uniform localization lemmas. For instance,
$\left(u_{n}^{T+1}\right)_{n}$ will denote the sequence obtained
by applying the lemmas to $\left(\bar{u}_{n}^{T}\right)_{n}$ and
to the interval $\left[0,T+1\right]$.

Despite this, we can exploit a monotonicity property of the bound
functions $N$ and $\eta$ provided by the uniform localization lemmas,
in order to end up with a locally bounded optimizing sequence $\left(v_{n}\right)_{n}$
and a ``pre-optimal'' function $v$ such that $v_{n}\rightharpoonup v$
in $L^{1}\left(\left[0,T\right]\right)$ for every $T>0$.

Then we prove the pointwise convergence of the states associated with
$\left(v_{n}\right)_{n}$. Furthermore, another - standard - diagonal
procedure is needed in order to extract from $\left(v_{n}\right)_{n}$
a sequence $\left(v_{n,n}\right)_{n}$ such that $\log v_{n,n}\rightharpoonup\log u_{*}$,
in every in $L^{1}\left(\left[0,T\right]\right)$, for a proper function
$u_{*}$. This is eventually proven to be an admissible and optimal
control, relying basically on dominated convergence combined with the following relations:
\begin{align*}
x\left(\cdot;v_{n}\right)\to x\left(\cdot;v\right) & \quad\mbox{pointwise in }\left[0,+\infty\right)\\
\log v_{n,n}\rightharpoonup\log u_{*} & \quad\mbox{in }L^{1}\left(\left[0,T\right]\right),\,\forall T>0\\
u_{*}\leq v & \quad\mbox{a.e. in }\left[0,+\infty\right),
\end{align*}
where $\mathrm{x}\left(\cdot;\mathrm{u}\right)$ denotes the trajectory
associated with the control $\mathrm{u}$.

These and other considerations serve as a semi-continuity
argument and allow to conclude the proof.\\

The scheme
\[
\mbox{uniform localization lemmas }\leftrightarrows\mbox{ "local" compactness }
\]
\[
\dashrightarrow\mbox{ two families diagonalization }\dashrightarrow\mbox{ one family diagonalization}
\]

can be considered a development and an improvement of the method introduced
in \cite{AB} and may be hopefully generalized to a scheme for obtaining
existence proofs, applicable to a wider class of infinite horizon
optimal control problems with non compact control space.\\

The model describes the dynamics of the accumulation of phosphorous
in the ecosystem of a shallow lake, from a optimal control theory
perspective. Precisely, the state equation expresses the (non-linear)
relationship between the farming activities near the lake, which are
responsible for the release of phosphorus, and the total amount of
phosphorous in the water, depending also on the natural production
and on the natural loss consisting of sedimentation, outflow and sequestration
in other biomass. The objective functional that is to be maximized,
represents the social benefit depending on the pollution released
by the farming activities, and takes into account the trade-offs between
the utility of the agricultural activities and the utility of a clear
lake.

Following \cite{Maler}, we can assert that the essential dynamics
of the eutrophication process can be modelled by the differential
equation:
\begin{equation}
\dot{P}\left(t\right)=-sP\left(t\right)+r\frac{P^{2}\left(t\right)}{m^{2}+P^{2}\left(t\right)}+L\left(t\right),\label{eq: prima}
\end{equation}
where $P$ is the amount of phosphorus in algae, $L$ is the input
of phosphorus (the \textquotedblleft{}loading\textquotedblright{}),
$s$ is the rate of loss consisting of sedimentation, outflow and
sequestration in other biomass, $r$ is the maximum rate of internal
loading and $m$ is the anoxic level.

After a change of variable and of time scale, we consider the normalized
equation
\[
\dot{x}\left(\tau\right)=-bx\left(\tau\right)+\frac{x^{2}\left(\tau\right)}{1+x^{2}\left(\tau\right)}+u\left(\tau\right),
\]
where $x\left(\cdot\right):=P\left(\cdot\right)/m$, $u\left(\cdot\right)=L\left(\cdot\right)/r$
and $b=sm/r$. Hence we see that the dynamics, as a function of the
state, shows a convex-concave behaviour.

In an economical analysis, the dynamics of pollution must be considered together with the social benefit of the different interest groups operating in the lake system. The social benefit obviously depends both on the status of the water and on the intensity of agricultural activities near the lake, which in a way can be measured by the amount of phosphorous released in the water.

Farmers have an interest in being able to increase the loading, so that the agricultural sector can grow without the need to invest in new technology in order to reduce emissions. On the other hand, groups such as fishermen, drinking water companies and any other industry making use of the water prefer a clear lake, and the same holds for people who use to spend leisure time in relation with the lake.
It is assumed that a community or country, balancing these different interests, can agree on a welfare
function of the form 
\[
\log\mathrm{u}-c\mathrm{x}^{2}\quad(c>0),
\]
in the sense that the lake has value as a ``waste sink'' for agriculture
$\log\mathrm{u}$, where $\mathrm{u}$ is the input of phosphorous
due to farming, and it provides ecological services that decrease
with the total amount of phosphorus $\mathrm{x}$ as $-c\mathrm{x}^{2}$.\\
Here we focus on the case of monotone dynamics, as a first, fundamental step foreshadowing further 
developments.

\section{Boundedness of the value function}
\begin{definition}
For every $x_{0}\geq0$ and every $u\in L_{loc}^{1}\left(\left[0,+\infty\right)\right)$
the function $t\to x\left(t;x_{0},u\right)$ is the solution to the
following Cauchy's Problem:

\begin{equation}
\begin{cases}
{\displaystyle \dot{x}\left(t\right)=F\left(x\left(t\right)\right)+u\left(t\right)} & \quad t\geq0\\
x\left(t\right)=x_{0}
\end{cases}\label{eq: eq stato}
\end{equation}

in the unknown $x\left(\cdot\right)$, where $F$ has the following
properties:
\begin{align*}
 & F\in\mathcal{C}^{1}\left(\mathbb{R},\mathbb{R}\right),\, F'\leq0\mbox{ in }\mathbb{R},\, F\left(0\right)=0,\,\lim_{x\to+\infty}F\left(x\right)=-\infty,\ {\displaystyle \lim_{x\to+\infty}F'\left(x\right):=-l<0},\\
 & \mbox{there exist }\bar{x}>0\mbox{ such that }F\mbox{ is convex in }\left[0,\bar{x}\right]\mbox{ and concave in }\left[\bar{x},+\infty\right)
\end{align*}
Moreover, we set $F'\left(0\right)<0$.\\
\\
For every $x_{0}\geq0$, the set of the admissible controls is:
\begin{eqnarray*}
\Lambda\left(x_{0}\right) & := & \left\{ u\in L_{loc}^{1}\left(\left[0,+\infty\right)\right)/u>0\quad\mbox{a.e. in }\left[0,+\infty\right)\right\} 
\end{eqnarray*}
and the \emph{objective functional} is defined by
\[
\mathcal{B}\left(x_{0};u\right)=\int_{0}^{+\infty}e^{-\rho t}\left[\log u\left(t\right)-cx^{2}\left(t;x_{0},u\right)\right]\mbox{d}t\quad\forall u\in\Lambda\left(x_{0}\right),
\]
where $\rho$ and $c$ are positive constants.

The \emph{value function} is
\[
V\left(x_{0}\right):=\sup_{u\in\Lambda\left(x_{0}\right)}\mathcal{B}\left(x_{0};u\right).
\]

\end{definition}
$ $
\begin{rem}
The Cauchy's problem $\eqref{eq: eq stato}$ has a unique global solution,
since the dynamics $F\left(\cdot\right)$ has (globally) bounded derivative.
We have
\[
-b_{0}x\leq F\left(x\right)\leq-bx+M,
\]
for some constants $b_{0},b,M>0$. This is easily proven setting $-b:=-l+\epsilon$
for $\epsilon>0$ sufficiently small, choose $b_{0}=F'\left(0\right)\wedge-b$
and use the assumption $F'\to-l$ at $+\infty$and the continuity
of $F$.
\end{rem}
$ $
\begin{rem}
\label{remark funzione h}Let $s_{1},s_{2}\geq0$, $u_{1},u_{2}\in L_{loc}^{1}\left(\left[0,+\infty\right),\mathbb{R}\right)$
and $t_{0}\geq0$.

Set $x_{1}=x\left(\cdot;s_{1},u_{1}\right)$, $x_{2}=x\left(\cdot;s_{2},u_{2}\right)$
and define:

\smallskip{}

\[
h\left(x_{1},x_{2}\right)\left(\tau\right):=\begin{cases}
{\displaystyle \frac{F\left(x_{1}\left(\tau\right)\right)-F\left(x_{2}\left(\tau\right)\right)}{x_{1}\left(\tau\right)-x_{2}\left(\tau\right)}} & \mbox{if }x_{1}\left(\tau\right)\neq x_{2}\left(\tau\right)\\
\\
F'\left(x_{1}\left(\tau\right)\right) & \mbox{if }x_{1}\left(\tau\right)=x_{2}\left(\tau\right).
\end{cases}
\]

\smallskip{}

Then $h\left(x_{1},x_{2}\right)$ is continuous, $-b_{0}\leq h\leq0$
and the following relation holds:
\begin{eqnarray}
\forall t\geq t_{0}:x_{1}\left(t\right)-x_{2}\left(t\right) & = & \exp\left(\int_{t_{0}}^{t}h\left(x_{1},x_{2}\right)\left(\tau\right)\mbox{d}\tau\right)\left(x_{1}\left(t_{0}\right)-x_{2}\left(t_{0}\right)\right)\nonumber \\
 &  & +\int_{t_{0}}^{t}\exp\left(\int_{s}^{t}h\left(x_{1},x_{2}\right)\left(\tau\right)\mbox{d}\tau\right)\left(u_{1}\left(s\right)-u_{2}\left(s\right)\right)\mbox{d}s.\nonumber \\
\label{eq: comparison}
\end{eqnarray}
In particular, taking $t_{0}=0$ and $s_{1}=s_{2}$:
\begin{equation}
\forall t\geq0:\, x_{1}\left(t\right)-x_{2}\left(t\right)=\int_{0}^{t}\exp\left(\int_{s}^{t}h\left(x_{1},x_{2}\right)\left(\tau\right)\mbox{d}\tau\right)\left(u_{1}\left(s\right)-u_{2}\left(s\right)\right)\mbox{d}s\label{eq: comparison da 0}
\end{equation}

Indeed, for every $t\geq t_{0}$:
\begin{eqnarray*}
\dot{x}_{1}\left(t\right)-\dot{x}_{2}\left(t\right) & = & F\left(x_{1}\left(t\right)\right)-F\left(x_{2}\left(t\right)\right)+u_{1}\left(t\right)-u_{2}\left(t\right)\\
 & = & h\left(x_{1},x_{2}\right)\left(t\right)\left[x_{1}\left(t\right)-x_{2}\left(t\right)\right]+u_{1}\left(t\right)-u_{2}\left(t\right).
\end{eqnarray*}
Multiplying both sides of this equation by $\exp\left(-\int_{t_{0}}^{t}h\left(x_{1},x_{2}\right)\left(\tau\right)\mbox{d}\tau\right)$
we obtain:
\begin{align*}
 & \frac{\mbox{d}}{\mbox{d}t}\left[\left(x_{1}\left(t\right)-x_{2}\left(t\right)\right)\exp\left(-\int_{t_{0}}^{t}h\left(x_{1},x_{2}\right)\left(\tau\right)\mbox{d}\tau\right)\right]\\
= & \exp\left(-\int_{t_{0}}^{t}h\left(x_{1},x_{2}\right)\left(\tau\right)\mbox{d}\tau\right)\left(u_{1}\left(t\right)-u_{2}\left(t\right)\right)\quad\forall t\geq t_{0}
\end{align*}
Fix $t\geq t_{0}$ and integrate between $t_{0}$ and $t$; then $\eqref{eq: comparison}$
is easily obtained.
\end{rem}
$ $
\begin{rem}
\label{remark comp ODE}Relation $\eqref{eq: comparison}$ implies
a well known comparison result, which in our case can be stated as
follows.

\emph{Let $s_{1},s_{2}\geq0$ and $u_{1},u_{2}\in L_{loc}^{1}\left(\left[0,+\infty\right),\mathbb{R}\right)$;
then for every $t_{0}\geq0$ and every $t_{1}\in\left(t_{0},+\infty\right]$,
if $u_{1}\geq u_{2}$ almost everywhere in $\left[t_{0},t_{1}\right]$
and $x\left(t_{0};s_{1},u_{1}\right)\geq x\left(t_{0};s_{2},u_{2}\right)$,
then}
\[
x\left(t;s_{1},u_{1}\right)\geq x\left(t;s_{2},u_{2}\right)\quad\forall t\in\left[t_{0},t_{1}\right].
\]

Moreover another classical comparison result implies that

\emph{for every $x_{0}\geq0$ and every $u\in L_{loc}^{1}\left(\left[0,+\infty\right)\right)$}:
\begin{eqnarray}
e^{-b_{0}t}\left(x_{0}+\int_{0}^{t}e^{b_{0}s}u\left(s\right)\mbox{d}s\right) & \leq & x\left(t;x_{0},u\right)\nonumber \\
 & \leq & e^{-bt}\left(x_{0}+\int_{0}^{t}e^{bs}\left(M+u\left(s\right)\right)\mbox{d}s\right).\label{eq: comparison per singola}
\end{eqnarray}

\end{rem}
$ $
\begin{rem}
\label{B div -infty}The objective functional is not constantly equal
to $-\infty$. As a trivial example, consider the control $u\equiv1\in\Lambda\left(x_{0}\right)$.
Then by $\eqref{eq: comparison per singola}$:
\[
0\leq x\left(t;x_{0},u\right)\leq e^{-bt}x_{0}+\left(M+1\right)\frac{1-e^{-bt}}{b}
\]
which implies
\[
x^{2}\left(t\right)\leq\left(x_{0}^{2}+\frac{\left(M+1\right)^{2}}{b^{2}}\right)e^{-2bt}+2\left(M+1\right)\frac{x_{0}}{b}e^{-bt}+\frac{\left(M+1\right)^{2}}{b^{2}}.
\]
Hence
\[
\mathcal{B}\left(u\right)=-c\int_{0}^{+\infty}e^{-\rho t}x^{2}\left(t;x_{0},u\right)\mbox{d}t>-\infty.
\]

\end{rem}
$ $
\begin{rem}
\label{oss: approx semplici}Let $u\in\Lambda\left(x_{0}\right)$
and let $\left(u_{n}\right)_{n}\subseteq L^{1}\left(\left[0,+\infty\right)\right)$
be a sequence of simple functions such that $u_{n}\uparrow u$ pointwise
in $\left[0,+\infty\right)$. Then
\[
\mathcal{B}\left(u\right)\leq\liminf_{n\to+\infty}\mathcal{B}\left(u_{n}\right).
\]
Indeed, for every $n\in\mathbb{N}$, $u_{n}>0$
almost everywhere in $\left[0,+\infty\right)$, so
$\left(e^{-\rho t}\log u_{n}\left(t\right)\right)_{n}\subseteq L^{1}\left(\left[0,+\infty\right)\right)$
and $e^{-\rho t}\log u_{n}\left(t\right)\uparrow e^{-\rho t}\log u\left(t\right)$
for almost every $t\geq0$. By monotone convergence we obtain:

\begin{eqnarray*}
\limsup_{n\to+\infty}\left[\mathcal{B}\left(u\right)-\mathcal{B}\left(u_{n}\right)\right] & = & \limsup_{n\to+\infty}\int_{0}^{+\infty}e^{-\rho t}\left[\log u\left(t\right)-\log u_{n}\left(t\right)-c\left(x^{2}\left(t\right)-x_{n}^{2}\left(t\right)\right)\right]\mbox{d}t\\
 & \leq & \lim_{n\to+\infty}\int_{0}^{+\infty}e^{-\rho t}\left[\log u\left(t\right)-\log u_{n}\left(t\right)\right]\mbox{d}t\\
 & = & 0,
\end{eqnarray*}
where the inequality holds since $0\leq x_{n}\leq x$ for every $n\in\mathbb{N}$,
by Remark $\ref{remark comp ODE}$.
\end{rem}
$ $
\begin{definition}
A sequence $\left(u_{n}\right)_{n\in\mathbb{N}}\subseteq\Lambda\left(x_{0}\right)$
is said to be \emph{maximizing at} $x_{0}$ if
\[
\lim_{n\to+\infty}\mathcal{B}\left(x_{0};u_{n}\right)=V\left(x_{0}\right).
\]

\end{definition}
$ $
\begin{proposition}
\label{prop: succ max} \emph{i)} The value function $V$:$\left[0,+\infty\right)\to\mathbb{R}$
satisfies:
\[
V\left(x_{0}\right)\leq\frac{1}{\rho}\log\left(\frac{\rho+b_{0}}{\sqrt{2ec}}\right)\quad\forall x_{0}\geq0.
\]

\emph{ii)} For every $x_{0}\geq0$, there exist constants $K_{1}\left(x_{0}\right),K_{2}\left(x_{0}\right)>0$
such that, for every $u\in\Lambda\left(x_{0}\right)$ belonging to
a maximizing sequence:
\begin{align}
 & \int_{0}^{+\infty}e^{-\rho t}u\left(t\right)\mbox{d}t\leq K_{1}\left(x_{0}\right),\label{eq: stima u massimizz}\\
 & \int_{0}^{+\infty}e^{-\rho t}x\left(t;x_{0},u\right)\left(t\right)\mbox{d}t\leq K_{2}\left(x_{0}\right).\label{eq: stima x massimizz}
\end{align}

\end{proposition}
Hereinafter we will often use the following weaker estimate relative
to a control $u\in\Lambda\left(x_{0}\right)$ belonging to a maximizing
sequence:
\begin{equation}
\int_{0}^{t}u\left(s\right)\mbox{d}s<K_{1}\left(x_{0}\right)e^{\rho t}\quad\forall t\geq0.\label{eq: stima cruciale massimizz}
\end{equation}

\begin{proof}

i) Let $x_{0}\geq0$, $u\in\Lambda\left(x_{0}\right)$, $x=x\left(\cdot;x_{0},u\right)$
and $\mathcal{B}\left(u\right)=\mathcal{B}\left(x_{0};u\right)$.

First assume that
\begin{equation}
\int_{0}^{+\infty}u\left(t\right)\mbox{d}t,\ \int_{0}^{+\infty}e^{-\rho t}u\left(t\right)\mbox{d}t<+\infty.\label{integrali u}
\end{equation}
We estimate the quantity
\[
\int_{0}^{+\infty}e^{-\rho t}x^{2}\left(t\right)\mbox{d}t
\]
in terms of the quantities in $\eqref{integrali u}$.

\emph{From above}: by $\eqref{eq: comparison per singola}$, we have
for every $t\geq0$:
\begin{eqnarray*}
0\leq x\left(t\right) & \leq & e^{-bt}x_{0}+\frac{M}{b}+e^{-bt}\int_{0}^{t}e^{bs}u\left(s\right)\mbox{d}s.
\end{eqnarray*}
Hence:
\begin{eqnarray}
x^{2}\left(t\right) & \leq & e^{-bt}\left(x_{0}\vee x_{0}^{2}\right)\left(1+\frac{2M}{b}\right)+\frac{M^{2}}{b^{2}}+e^{-2bt}\left(\int_{0}^{t}e^{bs}u\left(s\right)\mbox{d}s\right)^{2}\nonumber \\
 &  & +2\left(x_{0}\vee\frac{M}{b}\right)e^{-bt}\int_{0}^{t}e^{bs}u\left(s\right)\mbox{d}s.\label{eq: stima x^2 alto}
\end{eqnarray}
Focusing on the last two terms leads to the estimate
\begin{eqnarray}
\int_{0}^{+\infty}e^{-\rho t}e^{-2bt}\left(\int_{0}^{t}e^{bs}u\left(s\right)\mbox{d}s\right)^{2}\mbox{d}t & \leq & \int_{0}^{+\infty}e^{-\rho t}\left(\int_{0}^{t}u\left(s\right)\mbox{d}s\right)^{2}\mbox{d}t\nonumber \\
 & \leq & \frac{1}{\rho}\left(\int_{0}^{+\infty}u\left(s\right)\mbox{d}s\right)^{2}\label{eq: stima int x^2 pezzo 1}
\end{eqnarray}
and
\begin{eqnarray}
\int_{0}^{+\infty}e^{-\rho t}e^{-bt}\int_{0}^{t}e^{bs}u\left(s\right)\mbox{d}s\mbox{d}t & = & \int_{0}^{+\infty}e^{bs}u\left(s\right)\int_{s}^{+\infty}e^{-\left(\rho+b\right)t}\mbox{d}t\mbox{d}s\nonumber \\
 & = & \frac{1}{\rho+b}\int_{0}^{+\infty}e^{bs}u\left(s\right)e^{-\left(\rho+b\right)s}\mbox{d}s\nonumber \\
 & = & \frac{1}{\rho+b}\int_{0}^{+\infty}e^{-\rho s}u\left(s\right)\mbox{d}s.\label{eq: uguagl int x^2 pezzo 2}
\end{eqnarray}
By $\eqref{eq: stima x^2 alto}$, $\eqref{eq: stima int x^2 pezzo 1}$
and $\eqref{eq: uguagl int x^2 pezzo 2}$ we see that there exists
a constant $L\left(b,x_{0}\right)\geq0$ such that
\begin{eqnarray}
\int_{0}^{+\infty}e^{-\rho t}x^{2}\left(t\right)\mbox{d}t & \leq & L\left(b,x_{0}\right)+\frac{1}{\rho}\left(\int_{0}^{+\infty}u\left(t\right)\mbox{d}t\right)^{2}\nonumber \\
 &  & +2\left(x_{0}\vee\frac{M}{b}\right)\frac{1}{\rho+b}\int_{0}^{+\infty}e^{-\rho t}u\left(t\right)\mbox{d}t.\label{eq: stima int x^2 alto}
\end{eqnarray}
\emph{From below}: again by $\eqref{eq: comparison per singola}$:
\begin{eqnarray*}
\forall t\geq0:x\left(t\right) & \geq & e^{-b_{0}t}\left(x_{0}+\int_{0}^{t}e^{b_{0}s}u\left(s\right)\mbox{d}s\right)\\
 & \geq & e^{-b_{0}t}\int_{0}^{t}e^{b_{0}s}u\left(s\right)\mbox{d}s.
\end{eqnarray*}
Hence, since $t\to\rho e^{-\rho t}\mbox{d}t$ is a probability measure,
we have by Jensen's inequality:
\begin{eqnarray}
\int_{0}^{+\infty}e^{-\rho t}x^{2}\left(t\right)\mbox{d}t & \geq & \rho\left(\int_{0}^{+\infty}e^{-\rho t}x\left(t\right)\mbox{d}t\right)^{2}\nonumber \\
 & \geq & \rho\left(\int_{0}^{+\infty}e^{-\rho t}e^{-b_{0}t}\int_{0}^{t}e^{b_{0}s}u\left(s\right)\mbox{d}s\mbox{d}t\right)^{2}\nonumber \\
 & = & \frac{\rho}{\left(\rho+b_{0}\right)^{2}}\left(\int_{0}^{+\infty}e^{-\rho s}u\left(s\right)\mbox{d}s\right)^{2}\label{eq: stima int x^2 basso}
\end{eqnarray}
and the last equality holds by $\eqref{eq: uguagl int x^2 pezzo 2}$.

The finiteness of the integrals in $\eqref{integrali u}$ implies
that the application of Fubini's Theorem in $\eqref{eq: uguagl int x^2 pezzo 2}$
and in $\eqref{eq: stima int x^2 basso}$ are appropriate.

Relation $\eqref{eq: stima int x^2 basso}$ allows us to write down
the following estimate for $\mathcal{B}\left(u\right)$, using again
Jensen's inequality (in relation with the concave function log):
\begin{eqnarray}
\mathcal{B}\left(u\right) & = & \int_{0}^{+\infty}e^{-\rho t}\log u\left(t\right)\mbox{d}t-c\int_{0}^{+\infty}e^{-\rho t}x^{2}\left(t\right)\mbox{d}t\nonumber \\
 & \leq & \frac{1}{\rho}\log\left(\rho\int_{0}^{+\infty}e^{-\rho t}u\left(t\right)\mbox{d}t\right)-\frac{c}{\rho\left(\rho+b_{0}\right)^{2}}\left(\rho\int_{0}^{+\infty}e^{-\rho t}u\left(t\right)\mbox{d}t\right)^{2}\label{eq: stima B Paolo}\\
 & \leq & \frac{1}{\rho}\max_{z>0}\left(\log z-\frac{c}{\left(\rho+b_{0}\right)^{2}}z^{2}\right)=\frac{1}{\rho}\left(\log\frac{\rho+b_{0}}{\sqrt{2c}}-\frac{1}{2}\right)\label{eq: stima B Paolo max}\\
 & = & \frac{1}{\rho}\log\left(\frac{\rho+b_{0}}{\sqrt{2ec}}\right).
\end{eqnarray}
This holds under condition $\eqref{integrali u}$. In the opposite
case, that is to say $\int_{0}^{+\infty}e^{-\rho t}u\left(t\right)\mbox{d}t=+\infty$,
consider a sequence $\left(u_{n}\right)_{n\in\mathbb{N}}$ like in
Remark $\eqref{oss: approx semplici}$. Hence
\begin{eqnarray}
\mathcal{B}\left(u\right) & \leq & \liminf_{n\to+\infty}\mathcal{B}\left(u_{n}\right)\leq\liminf_{n\to+\infty}\frac{1}{\rho}\log\left(\rho\int_{0}^{+\infty}e^{-\rho t}u_{n}\left(t\right)\mbox{d}t\right)\nonumber \\
 &  & -\frac{c}{\rho\left(\rho+b_{0}\right)^{2}}\left(\rho\int_{0}^{+\infty}e^{-\rho t}u_{n}\left(t\right)\mbox{d}t\right)^{2}\nonumber \\
 & = & \lim_{z\to+\infty}\left(\frac{1}{\rho}\log z-\frac{c}{\rho\left(\rho+b_{0}\right)^{2}}z^{2}\right)=-\infty,\label{eq: B - infty}
\end{eqnarray}
since $\int_{0}^{+\infty}e^{-\rho t}u_{n}\left(t\right)\mbox{d}t\to\int_{0}^{+\infty}e^{-\rho t}u\left(t\right)\mbox{d}t$,
by monotone convergence.

In the intermediate case, that is to say
\[
\int_{0}^{+\infty}e^{-\rho t}u\left(t\right)\mbox{d}t<+\infty,\quad\int_{0}^{+\infty}u\left(t\right)\mbox{d}t=+\infty,
\]
let again $\left(u_{n}\right)_{n\in\mathbb{N}}$ be as in Remark $\eqref{oss: approx semplici}$.
We have:
\begin{eqnarray*}
\mathcal{B}\left(u\right) & \leq & \liminf_{n\to+\infty}\mathcal{B}\left(u_{n}\right)\leq\frac{1}{\rho}\log\left(\lim_{n\to+\infty}\rho\int_{0}^{+\infty}e^{-\rho t}u_{n}\left(t\right)\mbox{d}t\right)\\
 &  & -\frac{c}{\rho\left(\rho+b\right)^{2}}\left(\lim_{n\to+\infty}\rho\int_{0}^{+\infty}e^{-\rho t}u_{n}\left(t\right)\mbox{d}t\right)^{2}\\
 & = & \frac{1}{\rho}\log\left(\rho\int_{0}^{+\infty}e^{-\rho t}u\left(t\right)\mbox{d}t\right)-\frac{c}{\rho\left(\rho+b_{0}\right)^{2}}\left(\rho\int_{0}^{+\infty}e^{-\rho t}u\left(t\right)\mbox{d}t\right)^{2}\\
 & \leq & \frac{1}{\rho}\log\left(\frac{\rho+b_{0}}{\sqrt{2ec}}\right).
\end{eqnarray*}
Taking the sup among $u\in\Lambda\left(x_{0}\right)$, we see that
the same estimate holds for $V\left(x_{0}\right)$.

ii) Suppose that $u$ belongs to a maximizing sequence, and assume
that $\mathcal{B}\left(u\right)>V\left(x_{0}\right)-1$. Fix $\tilde{K}\left(x_{0}\right)\geq0$
such that 
\[
\frac{1}{\rho}\log z-\frac{c}{\rho\left(\rho+b_{0}\right)^{2}}z^{2}\leq V\left(x_{0}\right)-1\quad\forall z>\tilde{K}\left(x_{0}\right).
\]
We showed at point $i)$ that if $\int_{0}^{+\infty}e^{-\rho t}u\left(t\right)\mbox{d}t<+\infty$,
then relation $\eqref{eq: stima B Paolo}$, holds. Thus in this case
it must be 
\begin{equation}
\int_{0}^{+\infty}e^{-\rho t}u\left(t\right)\mbox{d}t\leq\frac{1}{\rho}\tilde{K}\left(x_{0}\right)=:K_{1}\left(x_{0}\right).\label{eq: unif bound int u}
\end{equation}
The case $\int_{0}^{+\infty}e^{-\rho t}u\left(t\right)\mbox{d}t=+\infty$
implies $\mathcal{B}\left(u\right)=-\infty$ by $\eqref{eq: B - infty}$,
and consequently must be excluded, since $u$ belongs to a maximizing
sequence (see Remark $\ref{B div -infty}$).

This proves relation $\eqref{eq: stima u massimizz}$.

In order to prove $\eqref{eq: stima x massimizz}$, observe that by
$\eqref{eq: comparison per singola}$ we have:
\begin{eqnarray*}
\int_{0}^{+\infty}e^{-\rho t}x\left(t\right)\mbox{d}t & \leq & \int_{0}^{+\infty}e^{-\rho t}\left\{ e^{-bt}x_{0}+\int_{0}^{t}e^{b\left(s-t\right)}\left(1+u\left(s\right)\right)\mbox{d}s\right\} \mbox{d}t\\
 & = & x_{0}\int_{0}^{+\infty}e^{-\left(\rho+b\right)t}\mbox{d}t+\int_{0}^{+\infty}e^{-\left(\rho+b\right)t}\int_{0}^{t}e^{bs}\mbox{d}s\mbox{d}t\\
 &  & +\int_{0}^{+\infty}e^{-\left(\rho+b\right)t}\int_{0}^{t}e^{bs}u\left(s\right)\mbox{d}s\mbox{d}t\\
 & = & \frac{x_{0}}{\rho+b}+\int_{0}^{+\infty}e^{bs}\int_{s}^{+\infty}e^{-\left(\rho+b\right)t}\mbox{d}t\mbox{d}s\\
 &  & +\int_{0}^{+\infty}u\left(s\right)e^{bs}\int_{s}^{+\infty}e^{-\left(\rho+b\right)t}\mbox{d}t\mbox{d}s\\
 & = & \frac{x_{0}}{\rho+b}+\frac{1}{\rho\left(\rho+b\right)}+\frac{1}{\rho+b}\int_{0}^{+\infty}e^{-\rho t}u\left(t\right)\mbox{d}t\\
 & \leq & \frac{x_{0}}{\rho+b}+\frac{1}{\rho\left(\rho+b\right)}+\frac{K_{1}\left(x_{0}\right)}{\rho+b}\\
 & =: & K_{2}\left(x_{0}\right)
\end{eqnarray*}
$ $
\end{proof}
$ $

\section{Uniform localization lemmas}
\begin{lemma}
\label{lem: loc alto}There exists a function $N:\left[0,+\infty\right)^{2}\to\left(0,+\infty\right)$,
continuous and strictly increasing in the second variable, such that:
for every $x_{0},T>0$ and for every $u\in\Lambda\left(x_{0}\right)$
belonging to a maximizing sequence, there exists a control $\tilde{u}^{T}\in\Lambda\left(x_{0}\right)$
satisfying:
\begin{align*}
 & \mathcal{B}\left(x_{0};\tilde{u}^{T}\right)\geq\mathcal{B}\left(x_{0};u\right)\\
 & \tilde{u}^{T}=u\wedge N\left(x_{0},T\right)\quad\mbox{a. e. in }\left[0,T\right].
\end{align*}
In particular, the norm $\left\Vert \tilde{u}^{T}\right\Vert _{L^{\infty}\left(\left[0,T\right]\right)}$
is bounded above by a quantity which does not depend on the original
control $u$.

Moreover, the state $x\left(\cdot;\tilde{u}^{T},x_{0}\right)$ associated
with the control $\tilde{u}^{T}$ satisfies
\[
x\left(\cdot;\tilde{u}^{T},x_{0}\right)\leq x\left(\cdot;u,x_{0}\right).
\]

Eventually, the bound function $N$ satisfies:
\begin{equation}
\lim_{T\to+\infty}Te^{-\rho T}\log N\left(x_{0},T\right)=0.\label{eq: logN a infinito}
\end{equation}
\end{lemma}
\begin{proof}
Fix $x_{0}$ and $T\geq0$. The equation
\begin{equation}
\log\beta+\beta b_{0}=-Tb_{0},\quad\mbox{\ensuremath{\beta}>0}\label{eq: def beta}
\end{equation}
has a unique solution, which is strictly less than $1$. Call this
solution $\beta_{T}$, and define 
\begin{equation}
N\left(x_{0},T\right):=K\left(x_{0}\right)\beta_{T}^{-2}e^{2\rho\left(T+\beta_{T}\right)},\label{eq: def N}
\end{equation}
where $K\left(x_{0}\right)=K_{1}\left(x_{0}\right)\vee1$ and $K_{1}\left(x_{0}\right)$
is the constant introduced in Proposition $\ref{prop: succ max}$.

Now fix $u\in\Lambda\left(x_{0}\right)$ such that $u$ belongs to
a maximizing sequence. If $u\leq N\left(x_{0},T\right)$ almost everywhere
in $\left[0,T\right]$, then set $\tilde{u}^{T}:=u$, and the proof
is over.

If there exists a non-negligible subset of $\left[0,T\right]$ in
which $u>N\left(x_{0},T\right)$ then define

\begin{align*}
 & \tilde{I}:=\int_{0}^{T}\left(u\left(t\right)-u\left(t\right)\wedge N\left(x_{0},T\right)\right)\mbox{d}t\\
 & \tilde{u}^{T}:=u\wedge N\left(x_{0},T\right)\cdot\chi_{\left[0,T\right]}+\left(u+\tilde{I}\right)\cdot\chi_{\left(T,T+\beta_{T}\right]}+u\cdot\chi_{\left(T+\beta_{T},+\infty\right)}.
\end{align*}
Obviously $\tilde{u}^{T}\in\Lambda\left(x_{0}\right)$, since $u\in\Lambda\left(x_{0}\right)$
and $N\left(x_{0},T\right)>0$.

First we prove that \textbf{
\begin{equation}
0\leq x\left(\cdot;\tilde{u}^{T},x_{0}\right)\leq x\left(\cdot;u,x_{0}\right)\quad\mbox{in }\left[0,+\infty\right)\label{eq: orbita alto meno di orig}
\end{equation}
}Clearly $x\left(\cdot;\tilde{u}^{T},x_{0}\right)\geq0$, by the admissibility
of $\tilde{u}^{T}$. For simplicity of notation we set $N=N\left(x_{0},T\right)$,
$\tilde{x}_{T}=x\left(\cdot;\tilde{u}^{T},x_{0}\right)$ and $x=x\left(\cdot;u,x_{0}\right)$.

Obviously $\tilde{x}_{T}\leq x$ in $\left[0,T\right]$, by Remark
$\ref{remark comp ODE}$.

Fix $t\in\left(T,T+\beta_{T}\right]$, and set $h:=h\left(\tilde{x}_{T},x\right)$,
like in Remark $\ref{remark funzione h}$. Hence:

\begin{eqnarray*}
\tilde{x}_{T}\left(t\right)-x\left(t\right) & = & \int_{0}^{T}\exp\left(\int_{s}^{t}h\mbox{d}\tau\right)\left(u\left(s\right)\wedge N-u\left(s\right)\right)\mbox{d}s\\
 &  & +\tilde{I}\int_{T}^{t}\exp\left(\int_{s}^{t}h\mbox{d}\tau\right)\mbox{d}s.
\end{eqnarray*}
The first addend is estimated in the following way:
\begin{eqnarray*}
\int_{0}^{T}\exp\left(\int_{s}^{t}h\mbox{d}\tau\right)\left(u\left(s\right)\wedge N-u\left(s\right)\right)\mbox{d}s & \leq & \int_{0}^{T}e^{\left(s-t\right)b_{0}}\left(u\left(s\right)\wedge N-u\left(s\right)\right)\mbox{d}s\\
 & \leq & e^{-tb_{0}}\int_{0}^{T}\left(u\left(s\right)\wedge N-u\left(s\right)\right)\mbox{d}s\\
 & \leq & e^{-\left(T+\beta_{T}\right)b_{0}}\int_{0}^{T}\left(u\left(s\right)\wedge N-u\left(s\right)\right)\mbox{d}s\\
 & = & -\tilde{I}e^{-\left(T+\beta_{T}\right)b_{0}}.
\end{eqnarray*}
Since $h\leq0$, the second addend is estimated from above by $\tilde{I}\beta_{T}$.

Thus we obtain:
\[
\tilde{x}_{T}\left(t\right)-x\left(t\right)\leq\tilde{I}\left(\beta_{T}-e^{-\left(T+\beta_{T}\right)b_{0}}\right),
\]
and the last quantity is zero, by definition of $\beta_{T}$.

This implies that $\tilde{x}_{T}\leq x$ also in $\left(T+\beta_{T},+\infty\right)$,
again by Remark $\ref{remark comp ODE}$. Hence, relation $\eqref{eq: orbita alto meno di orig}$
holds.

Now we estimate the ``logarithmic'' part of the difference between
$\mathcal{B}\left(x_{0};\tilde{u}^{T}\right)$ and $\mathcal{B}\left(x_{0};u\right)$.
By the concavity of log, we have:
\begin{eqnarray*}
 &  & \int_{0}^{+\infty}e^{-\rho t}\left(\log\tilde{u}^{T}\left(t\right)-\log u\left(t\right)\right)\mbox{d}t\\
 & = & \int_{0}^{T}e^{-\rho t}\left\{ \log\left(u\left(t\right)\wedge N\right)-\log u\left(t\right)\right\} \mbox{d}t\\
 &  & +\int_{T}^{T+\beta_{T}}e^{-\rho t}\left\{ \log\left(u\left(t\right)+\tilde{I}\right)-\log u\left(t\right)\right\} \mbox{d}t\\
 & \geq & \int_{0}^{T}e^{-\rho t}\left(u\left(t\right)\wedge N\right)^{-1}\left\{ u\left(t\right)\wedge N-u\left(t\right)\right\} \mbox{d}t\\
 &  & +\tilde{I}\int_{T}^{T+\beta_{T}}e^{-\rho t}\left(u\left(t\right)+\tilde{I}\right)^{-1}\mbox{d}t
\end{eqnarray*}
\begin{eqnarray}
 & = & \frac{1}{N}\int_{0}^{T}e^{-\rho t}\left\{ u\left(t\right)\wedge N-u\left(t\right)\right\} \mbox{d}t\nonumber \\
 &  & +\tilde{I}\int_{T}^{T+\beta_{T}}e^{-\rho t}\left(u\left(t\right)+\tilde{I}\right)^{-1}\mbox{d}t\nonumber \\
 & \geq & \frac{1}{N}\int_{0}^{T}\left(u\left(t\right)\wedge N-u\left(t\right)\right)\mbox{d}t\nonumber \\
 &  & +\tilde{I}\int_{T}^{T+\beta_{T}}e^{-\rho t}\left(u\left(t\right)+\tilde{I}\right)^{-1}\mbox{d}t\nonumber \\
 & = & \tilde{I}\left(\int_{T}^{T+\beta_{T}}e^{-\rho t}\left(u\left(t\right)+\tilde{I}\right)^{-1}\mbox{d}t-\frac{1}{N}\right).\label{eq: loc alto meglio 1}
\end{eqnarray}
Moreover, by Jensen's inequality:
\begin{eqnarray*}
\int_{T}^{T+\beta_{T}}e^{-\rho t}\left(u\left(t\right)+\tilde{I}\right)^{-1}\mbox{d}t & \geq & e^{-\rho\left(T+\beta_{T}\right)}\int_{T}^{T+\beta_{T}}\left(u\left(t\right)+\tilde{I}\right)^{-1}\mbox{d}t\\
 & \geq & \beta_{T}^{2}e^{-\rho\left(T+\beta_{T}\right)}\frac{1}{\int_{T}^{T+\beta_{T}}\left(u\left(t\right)+\tilde{I}\right)\mbox{d}t}\\
 & \geq & \beta_{T}^{2}e^{-\rho\left(T+\beta_{T}\right)}\frac{1}{\int_{T}^{T+\beta_{T}}u\left(t\right)\mbox{d}t+\tilde{I}}\\
 & \geq & \beta_{T}^{2}e^{-\rho\left(T+\beta_{T}\right)}\frac{1}{\int_{0}^{T+\beta_{T}}u\left(t\right)\mbox{d}t}
\end{eqnarray*}
where the penultimate inequality holds since $\beta_{T}<1$.

$ $

Now by Proposition $\ref{prop: succ max}$ we can complete this estimate
in the following way:
\begin{eqnarray}
\int_{T}^{T+\beta_{T}}e^{-\rho t}\left(u\left(t\right)+\tilde{I}\right)^{-1}\mbox{d}t & \geq & K\left(x_{0}\right)^{-1}\beta_{T}^{2}e^{-2\rho\left(T+\beta_{T}\right)}\nonumber \\
 & =: & \alpha\left(x_{0},T\right).\label{eq: loc alto meglio 2}
\end{eqnarray}
Observe that, by definition, $N\left(x_{0},T\right)=\alpha\left(x_{0},T\right)^{-1}$.
Hence, joining $\eqref{eq: loc alto meglio 1}$ with $\eqref{eq: loc alto meglio 2}$
we obtain 
\begin{eqnarray}
\int_{0}^{+\infty}e^{-\rho t}\left(\log\tilde{u}^{T}\left(t\right)-\log u\left(t\right)\right)\mbox{d}t & \geq & \tilde{I}\left(\alpha\left(x_{0},T\right)-\frac{1}{N\left(x_{0},T\right)}\right)=0.\label{eq: loc alto meglio 3}
\end{eqnarray}
This implies, by $\eqref{eq: orbita alto meno di orig}$:
\begin{eqnarray*}
\mathcal{B}\left(x_{0};\tilde{u}^{T}\right)-\mathcal{B}\left(x_{0};u\right) & = & \int_{0}^{+\infty}e^{-\rho t}\left(\log\tilde{u}^{T}\left(t\right)-\log u\left(t\right)\right)\mbox{d}t\\
 &  & -c\int_{0}^{+\infty}e^{-\rho t}\left\{ \tilde{x}_{T}^{2}\left(t\right)-x^{2}\left(t\right)\right\} \mbox{d}t\\
 & \geq & 0.
\end{eqnarray*}

Finally we prove the monotonicity of $N\left(x_{0},T\right)$ in $T$
.

First observe that $T\to\beta_{T}$ is clearly a strictly decreasing
function, since the function $\beta\to\log\beta+\beta b_{0}$ is strictly
increasing, and remembering equation $\eqref{eq: def beta}$.

Moreover, the function $T\to T+\beta_{T}$ is strictly increasing.
Indeed, set $f\left(x\right):=\log x+b_{0}x$ and let $\phi$ be the
inverse of $f$. Then $\beta_{T}=\phi\left(-Tb_{0}\right)$, and:
\[
\frac{\mbox{d}}{\mbox{d}T}\left(T+\beta_{T}\right)=1-b_{0}\phi'\left(-Tk\right)=1-\frac{b_{0}}{f'\left(\beta_{T}\right)}=1-\frac{b_{0}\beta_{T}}{1+b_{0}\beta_{T}}>0.
\]
This shows that $N\left(x_{0},\cdot\right)$ is strictly increasing.

Finally observe that:
\begin{equation}
\beta_{T}\sim e^{-Tb_{0}}\quad\mbox{for }T\to+\infty.\label{eq: andamento beta}
\end{equation}
Indeed, with $f$ defined as before, we have:
\[
\lim_{x\to0^{+}}\frac{f\left(x\right)}{\log x}=1.
\]
Hence $\phi\left(y\right)\sim e^{y}$ for $y\to-\infty$ and $\beta_{T}=\phi\left(-Tb_{0}\right)\sim e^{-Tb_{0}}$
for $T\to+\infty$.

It follows from $\eqref{eq: andamento beta}$ and $\eqref{eq: def N}$,
that:
\begin{eqnarray*}
Te^{-\rho T}\log N\left(x_{0},T\right) & = & Te^{-\rho T}\log K\left(x_{0}\right)+Te^{-\rho T}\log\left(\beta_{T}^{-2}\right)\\
 &  & +2\rho Te^{-\rho T}\left(T+\beta_{T}\right)\\
 & \sim & Te^{-\rho T}\log\left(\beta_{T}^{-2}\right)\\
 & \sim & 2T^{2}e^{-\rho T}b_{0}\quad\mbox{for }T\to+\infty.
\end{eqnarray*}
This shows that $\eqref{eq: logN a infinito}$ holds.
\end{proof}
$ $
\begin{lemma}
\label{lem: loc basso}There exists a function $\eta:\left[0,+\infty\right)^{2}\to\left(0,+\infty\right)$,
continuous and strictly decreasing in the second variable, with the
following property:
\begin{align*}
i)\  & \eta\left(x_{0},T\right)<N\left(x_{0},T\right)\quad\forall T>0
\end{align*}
where $N$ is the function defined in Lemma $\ref{lem: loc alto}$;\\

$ii)\,$ for every $x_{0}\geq0$ and every $T\geq1$, if $u\in\Lambda\left(x_{0}\right)$
belongs to a maximizing sequence, there exists $u^{T}\in\Lambda\left(x_{0}\right)$
such that
\begin{align*}
 & \mathcal{B}\left(x_{0};u^{T}\right)\geq\mathcal{B}\left(x_{0};u\right)\\
 & u^{T}=\left(u\wedge N\left(x_{0},T\right)\right)\vee\eta\left(x_{0},T\right)\quad\mbox{a. e. in }\left[0,T\right].
\end{align*}
In particular the norm $\left\Vert \log u^{T}\right\Vert _{L^{\infty}\left(\left[0,T\right]\right)}$
is bounded above by a quantity which does not depend on $u$.\end{lemma}
\begin{proof}
Fix $x_{0}$ and $u$ as in the hypothesis, and set $x:=x\left(\cdot;x_{0},u\right)$.
In order to define the function $\eta$, we preliminarily observe
that there obviously exits a number $L\left(x_{0}\right)>\rho$ such
that
\begin{equation}
e^{L\left(x_{0}\right)-\rho}-2c\rho^{-1}e^{-L\left(x_{0}\right)}\geq2cK_{2}\left(x_{0}\right).\label{eq: cond 1 L(x_0)}
\end{equation}
A simple computation shows that the function $T\to e^{\left(L\left(x_{0}\right)-\rho\right)T}-2c\rho^{-1}Te^{-L\left(x_{0}\right)T}$
is increasing if 
\begin{equation}
L\left(x_{0}\right)>\rho+\frac{2c}{\rho}.\label{eq: cond 2 L(x_0)}
\end{equation}
Now we now choose $L\left(x_{0}\right)$ satisfying $\eqref{eq: cond 1 L(x_0)}$
and $\eqref{eq: cond 2 L(x_0)}$ and we define 
\[
\eta\left(x_{0},T\right):=e^{-L\left(x_{0}\right)T}.
\]
Relation $i)$ follows from the fact that $N\left(x_{0},T\right)>1$;
moreover we have: 
\begin{equation}
e^{\left(L\left(x_{0}\right)-\rho\right)T}-2c\rho^{-1}Te^{-L\left(x_{0}\right)T}-2cK_{2}\left(x_{0}\right)\geq0\quad\forall T\geq1.\label{eq: per Bu_t meglio}
\end{equation}

Now fix $T\geq1$ and take $\tilde{u}^{T}$ as in Lemma $\ref{lem: loc alto}$.
Define $u^{T}:=\tilde{u}^{T}$ if $\tilde{u}^{T}\geq\eta\left(x_{0},T\right)$
almost everywhere in $\left[0,T\right]$, and 
\[
u^{T}:=\left(\tilde{u}^{T}\vee\eta\left(x_{0},T\right)\right)\chi_{\left[0,T\right]}+\tilde{u}^{T}\chi_{\left(T,+\infty\right)}
\]
if there exists a subset of $\left[0,T\right]$ of positive measure
where $\tilde{u}^{T}<\eta\left(x_{0},T\right)$. In this case define
also
\[
I:=\int_{0}^{T}\left[\tilde{u}^{T}\left(s\right)\vee\eta-\tilde{u}^{T}\left(s\right)\right]\mbox{d}s.
\]

We show that
\[
\mathcal{B}\left(x_{0};u^{T}\right)-\mathcal{B}\left(x_{0};\tilde{u}^{T}\right)\geq0,
\]
and the conclusion will follow from Lemma $\ref{lem: loc alto}$.

We provide two different estimates of the quantity $x\left(\cdot;x_{0},u_{T}\right)-x\left(\cdot;x_{0},\tilde{u}_{T}\right)$.
Set $x_{T}=x\left(\cdot;x_{0},u_{T}\right)$, $\tilde{x}_{T}=x\left(\cdot;x_{0},\tilde{u}_{T}\right)$,
$h=h\left(x_{T},\tilde{x}_{T}\right)$, $\eta=\eta\left(x_{0},T\right)$
and $N=N\left(x_{0},T\right)$ for simplicity of notation. Remembering
that $h\leq0$, we have, for every $t\in\left[0,T\right]$: 
\begin{eqnarray*}
x_{T}\left(t\right)-\tilde{x}_{T}\left(t\right) & = & \int_{0}^{t}e^{\int_{s}^{t}h\mbox{d}\tau}\left[u^{T}\left(s\right)-\tilde{u}^{T}\left(s\right)\right]\mbox{d}s\\
 & \leq & \int_{0}^{T}e^{\int_{s}^{t}h\mbox{d}\tau}\left[\tilde{u}^{T}\left(s\right)\vee\eta-\tilde{u}^{T}\left(s\right)\right]\mbox{d}s\\
 & \leq & I.
\end{eqnarray*}
The same estimate holds for $t>T$, since $u^{T}=\tilde{u}^{T}$ in
$\left(T,+\infty\right)$. Hence:
\begin{equation}
x_{T}-\tilde{x}_{T}\leq I\quad\mbox{in }\left[0,+\infty\right).\label{eq: stima x_T tilde - x}
\end{equation}

Moreover, since $\eta>0$:
\begin{eqnarray*}
I & = & \int_{0}^{T}\left[\tilde{u}^{T}\left(s\right)\vee\eta-\tilde{u}^{T}\left(s\right)\right]\mbox{d}s\\
 & = & \int_{\left[0,T\right]\cap\left\{ \tilde{u}^{T}\leq\eta\right\} }\left[\eta-\tilde{u}^{T}\left(s\right)\right]\mbox{d}s\\
 & \leq & T\eta.
\end{eqnarray*}

Hence
\begin{equation}
x_{T}-\tilde{x}_{T}\leq T\eta\quad\mbox{in }\left[0,+\infty\right).\label{eq: stima x_T tilde - x_T}
\end{equation}
By $\eqref{eq: stima x_T tilde - x}$ and $\eqref{eq: stima x_T tilde - x_T}$,
using the convexity relation $x^{2}-y^{2}\leq2x\left(x-y\right)$,
we obtain:
\begin{eqnarray*}
c\int_{0}^{+\infty}e^{-\rho t}\left[x_{T}^{2}\left(t\right)-\tilde{x}_{T}^{2}\left(t\right)\right]\mbox{d}t & \leq & 2c\int_{0}^{+\infty}e^{-\rho t}x_{T}\left(t\right)\left[x_{T}\left(t\right)-\tilde{x}_{T}\left(t\right)\right]\mbox{d}t\\
 & \leq & 2cI\int_{0}^{+\infty}e^{-\rho t}x_{T}\left(t\right)\mbox{d}t\\
 & = & 2cI\int_{0}^{+\infty}e^{-\rho t}\left[x_{T}\left(t\right)-\tilde{x}_{T}\left(t\right)\right]\mbox{d}t\\
 &  & +2cI\int_{0}^{+\infty}e^{-\rho t}\tilde{x}_{T}\left(t\right)\mbox{d}t\\
 & \leq & 2cIT\eta\int_{0}^{+\infty}e^{-\rho t}\mbox{d}t+2cI\int_{0}^{+\infty}e^{-\rho t}x\left(t\right)\mbox{d}t\\
 & \leq & I\left(2\frac{c}{\rho}T\eta+2cK_{2}\left(x_{0}\right)\right),
\end{eqnarray*}
where we also used $\eqref{eq: orbita alto meno di orig}$ and $\eqref{eq: stima x massimizz}$
(the trajectory $x\left(\cdot\right)$ is associated with a control
in a maximizing sequence).

Moreover:
\begin{eqnarray*}
\int_{0}^{+\infty}e^{-\rho t}\left(\log u^{T}\left(t\right)-\log\tilde{u}^{T}\left(t\right)\right)\mbox{d}t & = & \int_{0}^{T}e^{-\rho t}\left(\log\left(\tilde{u}^{T}\left(t\right)\vee\eta\right)-\log\tilde{u}^{T}\left(t\right)\right)\mbox{d}t\\
 & \geq & \int_{0}^{T}e^{-\rho t}\frac{1}{\tilde{u}^{T}\left(t\right)\vee\eta}\left(\tilde{u}^{T}\left(t\right)\vee\eta-\tilde{u}^{T}\left(t\right)\right)\mbox{d}t\\
 & = & \frac{1}{\eta}\int_{0}^{T}e^{-\rho t}\left(\tilde{u}^{T}\left(t\right)\vee\eta-\tilde{u}^{T}\left(t\right)\right)\mbox{d}t\\
 & \geq & \frac{e^{-\rho T}}{\eta}I.
\end{eqnarray*}

Joining the last two estimates leads to:
\begin{eqnarray*}
\mathcal{B}\left(x_{0};u^{T}\right)-\mathcal{B}\left(x_{0};\tilde{u}^{T}\right) & = & \int_{0}^{+\infty}e^{-\rho t}\left(\log u^{T}\left(t\right)-\log\tilde{u}^{T}\left(t\right)\right)\mbox{d}t\\
 &  & -c\int_{0}^{+\infty}e^{-\rho t}\left[x_{T}^{2}\left(t\right)-\tilde{x}_{T}^{2}\left(t\right)\right]\mbox{d}t\\
 & \geq & I\left(\frac{e^{-\rho T}}{\eta\left(x_{0},T\right)}-2\frac{c}{\rho}T\eta\left(x_{0},T\right)-2cK_{2}\left(x_{0}\right)\right)\\
\\
 & = & I\left(e^{\left(L\left(x_{0}\right)-\rho\right)T}-2c\rho^{-1}Te^{-L\left(x_{0}\right)T}-2cK_{2}\left(x_{0}\right)\right)\\
 & \geq & 0,
\end{eqnarray*}
where the last inequality holds by $\eqref{eq: per Bu_t meglio}$.
\end{proof}

\section{Diagonal procedures and functional convergence}

From this point on, the initial state $x_{0}\geq0$ is to be considered
fixed.
\begin{lemma}
\label{prop: new seq1}There exists a sequence $\left(v_{n}\right)_{n\in\mathbb{N}}$
and a function $v$ in $\Lambda\left(x_{0}\right)$ such that:
\begin{align}
 & \lim_{n\to+\infty}\mathcal{B}\left(x_{0};v_{n}\right)=V\left(x_{0}\right)\label{eq: new1 massimizz}\\
 & v_{n}\rightharpoonup v\mbox{ in }L^{1}\left(\left[0,T\right]\right)\quad\forall T>0\label{eq: new1 conv deb}\\
 & \forall T\in\mathbb{N}:\mbox{almost everywhere in}\left[0,T\right]:\nonumber \\
 & \forall n\geq T:\eta\left(x_{0},T\right)\leq v,v_{n}\leq N\left(x_{0},T\right)\label{eq: new1 bound}
\end{align}
where $N$, $\eta$ are the functions defined in Lemmas $\ref{lem: loc alto}$
and $\ref{lem: loc basso}$ .\end{lemma}
\begin{proof}
Set $\mathcal{B}=\mathcal{B}\left(x_{0};\cdot\right)$ and fix $\left(u_{n}\right)_{n\in\mathbb{N}}$
and such that
\[
\lim_{n\to+\infty}\mathcal{B}\left(u_{n}\right)=V\left(x_{0}\right).
\]

Set, for every $n\in\mathbb{N}$, $u_{n}^{1}$ as the function obtained
by applying Lemma $\ref{lem: loc basso}$ to $u_{n}$, for $T=1$.
Then
\begin{align*}
 & u_{n}^{1}=\left(u_{n}\wedge N\left(x_{0},1\right)\right)\vee\eta\left(x_{0},1\right)\quad\mbox{a.e. in }\left[0,1\right]\\
 & \mathcal{B}\left(u_{n}^{1}\right)\geq\mathcal{B}\left(u_{n}\right).
\end{align*}
Hence, as a consequence of the Dunford-Pettis criterion, there exists
a subsequence $\left(\overline{u}_{n}^{1}\right)_{n}$ of $\left(u_{n}^{1}\right)_{n}$
and a function $u^{1}\in L^{1}\left(\left[0,1\right]\right)$ such
that
\[
\overline{u}_{n}^{1}\rightharpoonup u^{1}\mbox{ in }L^{1}\left(\left[0,1\right]\right).
\]
Now apply Lemma $\ref{lem: loc basso}$ to the elements of the sequence
$\left(\overline{u}_{n}^{1}\right)_{n}$ in order to obtain a sequence
$\left(u_{n}^{2}\right)_{n}$ satisfying, for every $n\in\mathbb{N}$:
\begin{align*}
 & u_{n}^{2}=\left(\overline{u}_{n}^{1}\wedge N\left(x_{0},2\right)\right)\vee\eta\left(x_{0},2\right)\quad\mbox{a.e. in }\left[0,2\right]\\
 & \mathcal{B}\left(u_{n}^{2}\right)\geq\mathcal{B}\left(\overline{u}_{n}^{1}\right).
\end{align*}
Take, again by Dunford-Pettis, $\left(\overline{u}_{n}^{2}\right)_{n}$
extracted from $\left(u_{n}^{2}\right)_{n}$ and a function $u^{2}\in L^{1}\left(\left[0,2\right]\right)$
such that
\[
\overline{u}_{n}^{2}\rightharpoonup u^{2}\mbox{ in }L^{1}\left(\left[0,2\right]\right).
\]
Iterating this process we define families $\left(\overline{u}_{n}^{T}\right)_{n}$,
$\left(u_{n}^{T}\right)_{n}$, $\sigma_{T}$ ($T\in\mathbb{N}$) such
that the $\sigma_{T}$'s are strictly increasing with $\sigma_{T}\geq Id$
, satisfying for every $T,n\in\mathbb{N}$:
\begin{align}
 & \overline{u}_{n}^{T}=u_{\sigma_{T}\left(n\right)}^{T}\label{eq: provv subseq}\\
 & u_{n}^{T}=\left(\overline{u}_{n}^{T-1}\wedge N\left(x_{0},T\right)\right)\vee\eta\left(x_{0},T\right)\quad\mbox{a.e. in }\left[0,T\right]\label{eq: provv bound}\\
 & \mathcal{B}\left(u_{n}^{T}\right)\geq\mathcal{B}\left(\overline{u}_{n}^{T-1}\right)\label{eq: provv meglio}\\
 & \overline{u}_{n}^{T}\rightharpoonup u^{T}\mbox{ in }L^{1}\left(\left[0,T\right]\right).\label{eq: provv conv}
\end{align}
Fix $T\in\mathbb{N}$. The sequence $\left(\overline{u}{}_{n}^{T}\right)_{n}$
coincides, almost everywhere in $\left[0,T-1\right]$, with a sequence
that is extracted from $\left(\overline{u}{}_{n}^{T-1}\right)_{n}$.
Indeed, for every $n\in\mathbb{N}$:
\begin{eqnarray*}
\bar{u}_{n}^{T} & = & u_{\sigma_{T}\left(n\right)}^{T}\overset{{\scriptstyle a.e.\, in\,}{\scriptscriptstyle \left[0,T\right]}}{=}\left(\overline{u}_{\sigma_{T}\left(n\right)}^{T-1}\wedge N\left(x_{0},T\right)\right)\vee\eta\left(x_{0},T\right)\\
 & \overset{{\scriptstyle a.e.\, in\,}{\scriptscriptstyle \left[0,T-1\right]}}{=} & \overline{u}_{\sigma_{T}\left(n\right)}^{T-1}.
\end{eqnarray*}
The last equality holds since applying recursively (in $T$) relation
$\eqref{eq: provv bound}$ together with relation $\eqref{eq: provv subseq}$
gives $\overline{u}_{\sigma_{T}\left(n\right)}^{T-1}\in\left[\eta\left(x_{0},T-1\right),N\left(x_{0},T-1\right)\right]$;
then observe that by Lemmas $\ref{lem: loc alto}$ and $\ref{lem: loc basso}$
the function $\eta\left(x_{0},\cdot\right)$ is decreasing and the
function $N\left(x_{0},\cdot\right)$ is increasing.

Hence $u^{T-1}=u^{T}$ almost everywhere in $\left[0,T-1\right]$,
by the essential uniqueness of the weak limit.

Hence, defining 
\[
\forall t\geq0:v\left(t\right):=u^{\left[t\right]+1}\left(t\right)
\]
we obtain $v=u^{T}$ almost everywhere in $\left[0,T\right]$ and
\begin{equation}
\forall T\in\mathbb{N}:\overline{u}_{n}^{T}\rightharpoonup v\quad\mbox{in }L^{1}\left[0,T\right].\label{eq: provv conv 2}
\end{equation}

Repeating the previous argument, we see that for every $T,n\in\mathbb{N}$:
\begin{eqnarray*}
\bar{u}_{n}^{T} & \overset{{\scriptstyle a.e.\, in\,}{\scriptscriptstyle \left[0,T-1\right]}}{=} & \overline{u}_{\sigma_{T}\left(n\right)}^{T-1}\\
 & \overset{{\scriptstyle a.e.\, in\,}{\scriptscriptstyle \left[0,T-2\right]}}{=} & \overline{u}_{\sigma_{T-1}\circ\sigma_{T}\left(n\right)}^{T-2}\\
 & \dots\\
 & \overset{{\scriptstyle a.e.\, in\,}{\scriptscriptstyle \left[0,T-j\right]}}{=} & \overline{u}_{\sigma_{T-j+1}\circ\dots\circ\sigma_{T}\left(n\right)}^{T-j}.
\end{eqnarray*}
Observe that $\left(\overline{u}_{\sigma_{T-j+1}\circ\dots\circ\sigma_{T}\left(n\right)}^{T-j}\right)_{n}$
is a subsequence of $\left(\overline{u}_{n}^{T-j}\right)_{n}$ since
the composition $\sigma_{T-j+1}\circ\dots\circ\sigma_{T}$ is strictly
increasing and satisfies
\[
\sigma_{T-j+1}\circ\dots\circ\sigma_{T}\left(n\right)\geq n\quad\forall n\in\mathbb{N}.
\]
Hence, inverting the quantifiers ``$\forall n\in\mathbb{N}$'' and
``a.e. in $\left[0,T-j\right]$'', we see that $\left(\overline{u}_{n}^{T}\right)_{n}$
coincides, almost everywhere in $\left[0,T-j\right]$ with a subsequence
of $\left(\overline{u}_{n}^{T-j}\right)_{n}$, for every $T\in\mathbb{N}$
and $j=1,\dots,T-1$.

This implies that for every $T\in\mathbb{N}$ the sequence $\left(v_{n}\right)_{n\geq T}$
defined by $v_{n}:=\overline{u}_{n}^{n}$ coincides with a subsequence
of $\left(\overline{u}_{n}^{T}\right)_{n\geq1}$, almost everywhere
in $\left[0,T\right]$. Hence
\begin{align}
 & \forall T\in\mathbb{N}:\mbox{almost everywhere in }\left[0,T\right]:\nonumber \\
 & \forall n\geq T:\eta\left(x_{0},T\right)\leq v_{n}\leq N\left(x_{0},T\right).\label{eq: seq new1 bound}
\end{align}

and
\begin{align*}
 & v_{n}\rightharpoonup v\mbox{ in }L^{1}\left(\left[0,T\right]\right)\quad\forall T\in\mathbb{N},
\end{align*}
by $\eqref{eq: provv bound}$ and $\eqref{eq: provv conv 2}$.

The extension to every $T>0$ is straightforward, so we obtain $\eqref{eq: new1 conv deb}$.
Now fix $T>0$; a well known property of the weak convergence implies
that
\begin{equation}
\liminf_{n\to+\infty}v_{n}\left(t\right)\leq v\left(t\right)\leq\limsup_{n\to+\infty}v_{n}\left(t\right)\mbox{ for almost every }t\in\left[0,T\right].\label{eq: propr conv deb}
\end{equation}
Considering the intersection between the subsets of $\left[0,T\right]$
where relations $\eqref{eq: seq new1 bound}$ and $\eqref{eq: propr conv deb}$
hold, we obtain $\eqref{eq: new1 bound}$.

In order to prove $\eqref{eq: new1 massimizz}$, observe that
\begin{eqnarray*}
\mathcal{B}\left(v_{n}\right) & = & \mathcal{B}\left(u_{\sigma_{n}\left(n\right)}^{n}\right)\geq\mathcal{B}\left(\overline{u}_{\sigma_{n}\left(n\right)}^{n-1}\right)\\
 & = & \mathcal{B}\left(u_{\sigma_{n-1}\circ\sigma_{n}\left(n\right)}^{n-1}\right)\geq\dots\geq\mathcal{B}\left(u_{\sigma_{n-2}\circ\sigma_{n-1}\circ\sigma_{n}\left(n\right)}^{n-2}\right)\\
 & \geq & \dots\geq\mathcal{B}\left(u_{\sigma_{1}\circ\dots\circ\sigma_{n}\left(n\right)}^{1}\right)\geq\mathcal{B}\left(u_{\sigma_{1}\circ\dots\circ\sigma_{n}\left(n\right)}\right).
\end{eqnarray*}
Fix $\epsilon>0$ and $n_{\epsilon}\in\mathbb{N}$ such that $V\left(x_{0}\right)-\mathcal{B}\left(u_{n}\right)<\epsilon$
for $n\geq n_{\epsilon}$; since $\sigma_{1}\circ\dots\circ\sigma_{m}\geq Id$,
we have
\[
V\left(x_{0}\right)-\mathcal{B}\left(v_{n}\right)<\epsilon\quad\forall n\geq n_{\epsilon}.
\]
\end{proof}
\begin{proposition}
\label{prop: conv punt orbits}Let $v_{n}$ ($n\in\mathbb{N}$) and
$v$ be as in Proposition $\ref{prop: new seq1}$, and let $x_{n}:=x\left(\cdot;x_{0},v_{n}\right)$
and $x:=x\left(\cdot;x_{0},v\right)$ be the associated trajectories
starting at $x_{0}$. Then
\[
x_{n}\to x\quad\mbox{pointwise in }\left[0,+\infty\right).
\]
\end{proposition}
\begin{proof}
Fix $T>0$. By $\eqref{eq: new1 bound}$ in Proposition $\ref{prop: new seq1}$
and by Remark $\ref{remark comp ODE}$, $v$ is admissible and the
following uniform estimate holds: 
\begin{equation}
\left|x-x_{n}\right|\leq x\left(\cdot;x_{0},N\left(x_{0},T\right)\right)\quad\mbox{in }\left[0,T\right],\,\forall n\in\mathbb{N}.\label{eq: stima orbit unif}
\end{equation}

Now fix $t\in\left[0,T\right]$ and $n\in\mathbb{N}$. Subtracting
the state equation for $x$ from the state equation for $x_{n}$,
we obtain, for every $s\in\left[0,t\right]$:
\begin{eqnarray*}
\dot{x_{n}}\left(s\right)-\dot{x}\left(s\right) & = & F\left(x_{n}\left(s\right)\right)-F\left(x\left(s\right)\right)+v_{n}\left(s\right)-v\left(s\right)\\
 & = & h_{n}\left(s\right)\left[x_{n}\left(s\right)-x\left(s\right)\right]+v_{n}\left(s\right)-v\left(s\right),
\end{eqnarray*}
where $h_{n}:=h\left(x_{n},x\right)$ is the function defined in Remark
$\ref{remark funzione h}$.

Integrating both sides of this equation between $0$ and $t$, then
taking absolute values leads to:
\begin{eqnarray}
\left|x_{n}\left(t\right)-x\left(t\right)\right| & \leq\int_{0}^{t}\left|h_{n}\left(s\right)\right|\left|x_{n}\left(s\right)-x\left(s\right)\right|\mbox{d}s & +\left|\int_{0}^{t}\left[v_{n}\left(s\right)-v\left(s\right)\right]\mbox{d}s\right|.\label{eq: per conv qo orbite}
\end{eqnarray}
Observe that, for every $s\in\left[0,t\right]$:
\begin{eqnarray*}
\left|h_{n}\left(s\right)\right|\left|x_{n}\left(s\right)-x\left(s\right)\right| & \leq & b_{0}x\left(s;x_{0},N\left(x_{0},T\right)\right),
\end{eqnarray*}
by Remark $\ref{remark funzione h}$ and by $\eqref{eq: stima orbit unif}$.

Since the function on the right hand side obviously belongs to $L^{1}\left(\left[0,t\right]\right)$,
passing to the limsup in $\eqref{eq: per conv qo orbite}$ and remembering
$\eqref{eq: new1 conv deb}$, we obtain by Dominated Convergence:
\begin{eqnarray}
\limsup_{n\to+\infty}\left|x_{n}\left(t\right)-x\left(t\right)\right| & \leq & \limsup_{n\to+\infty}\int_{0}^{t}\left|h_{n}\left(s\right)\right|\left|x_{n}\left(s\right)-x\left(s\right)\right|\mbox{d}s\nonumber \\
 & = & \int_{0}^{t}\limsup_{n\to+\infty}\left|h_{n}\left(s\right)\right|\left|x_{n}\left(s\right)-x\left(s\right)\right|\mbox{d}s\label{eq: conv orbits 1}\\
 & \leq & b_{0}\int_{0}^{t}\limsup_{n\to+\infty}\left|x_{n}\left(s\right)-x\left(s\right)\right|\mbox{d}s.\nonumber 
\end{eqnarray}
Hence by Gronwall's inequality:
\[
\limsup_{n\to+\infty}\left|x_{n}\left(t\right)-x\left(t\right)\right|=0,
\]
for every $t\in\left[0,T\right]$. This is equivalent to

\[
\lim_{n\to+\infty}x_{n}=x\quad\mbox{in }\left[0,T\right],
\]

which proves the thesis, since $T>0$ is generic.
\end{proof}
$ $
\begin{lemma}
\label{lem: new seq 2} Take $\left(v_{n}\right)_{n\in\mathbb{N}}$
and $v$ as in Lemma $\ref{prop: new seq1}$. There exists a sequence\\
$\left(v_{n,n}\right)_{n\in\mathbb{N}}$, extracted from $\left(v_{n}\right)_{n\in\mathbb{N}}$,
and a function $u_{*}\in\Lambda\left(x_{0}\right)$, satisfying, for
every $T>0$:

\begin{align}
 & \log v_{n,n}\rightharpoonup\log u_{*}\quad\mbox{in }L^{1}\left(\left[0,T\right]\right)\label{eq: new2 conv deb}\\
 & \eta\left(x_{0},T\right)\leq u_{*}\leq N\left(x_{0},T\right)\quad\mbox{ a. e. in }\left[0,T\right].\label{eq: u_star bound}\\
 & 0\leq x\left(\cdot;x_{0},u_{*}\right)\leq x\left(\cdot;x_{0},v\right)\quad\mbox{ in}\left[0,+\infty\right).\label{eq: x_star minore x}
\end{align}
\end{lemma}
\begin{proof}
We conduct ``standard'' diagonalization on the sequence $\left(\log v_{n}\right)_{n\in\mathbb{N}}$.
Observe that this sequence, by $\eqref{eq: new1 bound}$, is also
uniformly bounded in the $L_{\left[0,1\right]}^{\infty}$ norm. Precisely,
for any $n\in\mathbb{N}$:
\[
\log\eta\left(x_{0},1\right)\leq\log v_{n}\leq\log N\left(x_{0},1\right)\quad\mbox{a.e. in }\left[0,1\right].
\]
Hence by the Dunford-Pettis criterion there exists a function $f^{1}\in L^{1}\left(\left[0,1\right]\right)$
and a sequence $\left(v_{n,1}\right)_{n}$ extracted form $\left(v_{n}\right)$
such that
\[
\log v_{n,1}\rightharpoonup f^{1}\quad\mbox{in }L^{1}\left(\left[0,1\right]\right).
\]
Again by $\eqref{eq: new1 bound}$, $\left(v_{n,1}\right)_{n}$ satisfies,
for every $n\in\mathbb{N}$:
\[
\log\eta\left(x_{0},2\right)\leq\log v_{n,1}\leq\log N\left(x_{0},2\right)\quad\mbox{a.e. in }\left[0,2\right];
\]
therefore there exist $f^{2}\in L^{1}\left(\left[0,2\right]\right)$
and $\left(v_{n,2}\right)_{n}$ extracted from $\left(v_{n,1}\right)_{n}$
such that
\[
\log v_{n,2}\rightharpoonup f^{2}\quad\mbox{in }L^{1}\left(\left[0,2\right]\right),
\]
and so on. This shows that there exists a function $f\in L_{loc}^{1}\left(\left[0,+\infty\right)\right)$
satisfying, together with the diagonal sequence $\left(v_{n,n}\right)_{n}$,
for every $T>0$:
\begin{align*}
 & \log v_{n,n}\rightharpoonup f\quad\mbox{in }L^{1}\left(\left[0,T\right]\right)\\
 & \log\eta\left(x_{0},T\right)\leq\log v_{n,n}\leq\log N\left(x_{0},T\right)\ \mbox{a.e. in }\left[0,T\right],\forall n\geq T.
\end{align*}

Define $u_{*}:=e^{f}$; then relations $\eqref{eq: new2 conv deb}$
and $\eqref{eq: u_star bound}$ are easy consequences of this definition
and of the properties of the weak convergence.

In order to prove $\eqref{eq: x_star minore x}$, we first observe
that, obviously, $x\left(\cdot;x_{0},u_{*}\right)\geq0$. Fix $0<t_{0}<t_{1}<T$
and let $t_{0}$ be a Lebesgue point for both $\log u_{*}$ and $v$.
By Jensen's inequality we have, for every $n\in\mathbb{N}$: 
\[
\frac{\int_{t_{0}}^{t_{1}}\log v_{n,n}\left(s\right)\mbox{d}s}{t_{1}-t_{0}}\leq\log\left(\frac{\int_{t_{0}}^{t_{1}}v_{n,n}\left(s\right)\mbox{d}s}{t_{1}-t_{0}}\right);
\]
since $\left(v_{n,n}\right)_{n}$ is a subsequence of $\left(v_{n}\right)_{n}$,
passing to the limit for $n\to+\infty$ in the previous relation,
we obtain by $\eqref{eq: new1 conv deb}$ and $\eqref{eq: new2 conv deb}$:
\[
\frac{\int_{t_{0}}^{t_{1}}\log u_{*}\left(s\right)\mbox{d}s}{t_{1}-t_{0}}\leq\log\left(\frac{\int_{t_{0}}^{t_{1}}v\left(s\right)\mbox{d}s}{t_{1}-t_{0}}\right).
\]
Passing now to the limit for $t_{1}\to t_{0}$ yields to $\log u_{*}\left(t_{0}\right)\leq\log v\left(t_{0}\right)$.
By the Lebesgue Point Theorem, $t_{0}$ is a generic element of a
full measure subset of $\left[0,T\right]$. This implies $\eqref{eq: x_star minore x}$,
by Remark $\ref{remark comp ODE}$.
\end{proof}
$ $

A simple integration by parts provides the following decomposition
of the objective functional:
\begin{eqnarray*}
\forall\mathrm{u}\in\Lambda\left(x_{0}\right):\ \mathcal{B}\left(x_{0};\mathrm{u}\right) & = & \int_{0}^{+\infty}e^{-\rho t}\left(\log\mathrm{u}\left(t\right)-cx^{2}\left(t\right)\right)\mbox{d}t\\
 & = & \int_{0}^{+\infty}e^{-\rho t}\log\mathrm{u}\left(t\right)\mbox{d}t-c\int_{0}^{+\infty}e^{-\rho t}x^{2}\left(t\right)\mbox{d}t\\
 & = & \lim_{T\to+\infty}e^{-\rho T}\int_{0}^{T}\log\mathrm{u}\left(s\right)\mbox{d}s+\\
 &  & \rho\int_{0}^{+\infty}e^{-\rho t}\left(\int_{0}^{t}\log\mathrm{u}\left(s\right)\mbox{d}s-\frac{c}{\rho}x^{2}\left(t\right)\right)\mbox{d}t\\
 & =: & \lim_{T\to+\infty}e^{-\rho T}\int_{0}^{T}\log\mathrm{u}\left(t\right)\mbox{d}t+\mathcal{B}_{1}\left(x_{0};\mathrm{u}\right)
\end{eqnarray*}

where
\[
\mathcal{B}_{1}\left(x_{0};\mathrm{u}\right):=\rho\int_{0}^{+\infty}e^{-\rho t}\left(\int_{0}^{t}\log\mathrm{u}\left(s\right)\mbox{d}s-\frac{c}{\rho}x^{2}\left(t;x_{0},\mathrm{u}\right)\right)\mbox{d}t.
\]
With this notation, we prove the final step.

$ $
\begin{corollary}
The control $u_{*}$ defined in Lemma $\ref{lem: new seq 2}$ is optimal
at $x_{0}$, and 
\[
u_{*}\in L_{loc}^{\infty}\left(\left[0,+\infty\right)\right).
\]
\end{corollary}
\begin{proof}
Obviously $u_{*}\in L_{loc}^{\infty}\left([0,+\infty)\right)$, by
$\eqref{eq: u_star bound}$. Observe that, by Jensen's inequality
and by Proposition $\ref{prop: succ max}$, for every $n\in\mathbb{N}$
and $t>0$:
\begin{eqnarray}
e^{-\rho t}\int_{0}^{t}\log v_{n,n}\left(s\right)\mbox{d}s & \leq & te^{-\rho t}\log\left(\frac{\int_{0}^{t}v_{n,n}\left(s\right)\mbox{d}s}{t}\right)\nonumber \\
 & \leq & te^{-\rho t}\log\left(K\left(x_{0}\right)e^{\rho t}\right)-te^{-\rho t}\log\left(t\right).\label{eq: stima int-log-v_n,n}
\end{eqnarray}
This implies that $\lim_{t\to+\infty}e^{-\rho t}\int_{0}^{t}\log v_{n,n}\left(s\right)\mbox{d}s\leq0$
and consequently 
\begin{equation}
\mathcal{B}\left(x_{0};v_{n,n}\right)\leq\mathcal{B}_{1}\left(x_{0};v_{n,n}\right).\label{eq: B min B_1 ultimo}
\end{equation}
Moreover
\begin{align}
 & \int_{0}^{+\infty}\Bigl(te^{-\rho t}\log\left(K\left(x_{0}\right)e^{\rho t}\right)-te^{-\rho t}\log\left(t\right)\Bigr)\mbox{d}t\nonumber \\
\leq & \int_{0}^{1}te^{-\rho t}\log\left(K\left(x_{0}\right)e^{\rho t}\right)\mbox{d}t-\int_{0}^{1}te^{-\rho t}\log\left(t\right)\mbox{d}t\nonumber \\
 & +\int_{1}^{+\infty}te^{-\rho t}\log\left(K\left(x_{0}\right)e^{\rho t}\right)\mbox{d}t\ <\ +\infty.\label{eq: bound funzione somm}
\end{align}
 Set $x_{n,n}:=x\left(\cdot;x_{0},v_{n,n}\right)$, $x:=\left(\cdot;x_{0},v\right)$
and $x_{*}:=\left(\cdot;x_{0},u_{*}\right)$. Relations $\eqref{eq: stima int-log-v_n,n}$
and $\eqref{eq: bound funzione somm}$ imply that the hypotheses of
Lemma $\ref{lem: reverse fatou}$ are satisfied for the integral
\[
\int_{0}^{\infty}e^{-\rho t}\left(\int_{0}^{t}\log v_{n,n}\left(s\right)\mbox{d}s-\frac{c}{\rho}x_{n,n}^{2}\left(t\right)\right)\mbox{d}t.
\]

Combining this result with relations $\eqref{eq: B min B_1 ultimo}$,$\eqref{eq: new2 conv deb}$,
$\eqref{eq: x_star minore x}$ and with Proposition $\ref{prop: conv punt orbits}$
we obtain:

\begin{eqnarray*}
V\left(x_{0}\right) & = & \lim_{n\to+\infty}\mathcal{B}\left(x_{0};v_{n,n}\right)\leq\lim_{n\to+\infty}\mathcal{B}_{1}\left(x_{0};v_{n,n}\right)\\
 & = & \rho\lim_{n\to+\infty}\int_{0}^{+\infty}e^{-\rho t}\left(\int_{0}^{t}\log v_{n,n}\left(s\right)\mbox{d}s-\frac{c}{\rho}x_{n,n}^{2}\left(t\right)\right)\mbox{d}t\\
 & \leq & \rho\int_{0}^{+\infty}e^{-\rho t}\limsup_{n\to+\infty}\left(\int_{0}^{t}\log v_{n,n}\left(s\right)\mbox{d}s-\frac{c}{\rho}x_{n,n}^{2}\left(t\right)\right)\mbox{d}t\\
 & = & \rho\int_{0}^{+\infty}e^{-\rho t}\left(\int_{0}^{t}\log u_{*}\left(s\right)\mbox{d}s-\frac{c}{\rho}x^{2}\left(t\right)\right)\mbox{d}t\\
 & \leq & \rho\int_{0}^{+\infty}e^{-\rho t}\left(\int_{0}^{t}\log u_{*}\left(s\right)\mbox{d}s-\frac{c}{\rho}x_{*}^{2}\left(t\right)\right)\mbox{d}t\\
 & = & \mathcal{B}_{1}\left(x_{0};u_{*}\right).
\end{eqnarray*}

Finally observe that by $\eqref{eq: u_star bound}$, for every $t\geq0$:
\[
te^{-\rho t}\log\eta\left(x_{0},t+1\right)\leq e^{-\rho t}\int_{0}^{t}\log u_{*}\left(s\right)\mbox{d}s\leq te^{-\rho t}\log N\left(x_{0},t+1\right),
\]
which implies that the estimated quantity vanishes for $t\to+\infty$,
since $\eta\left(x_{0},t\right)=e^{-L\left(x_{0}\right)t}$ and by
$\eqref{eq: logN a infinito}$.

Hence $\mathcal{B}_{1}\left(x_{0};u_{*}\right)=\mathcal{B}\left(x_{0};u_{*}\right)$,
and this concludes the proof.
\end{proof}

\section*{Appendix}
\addcontentsline{toc}{section}{Appendix}

\begin{lemma}
\label{lem: reverse fatou}Let $\left(E,\sigma,\mu\right)$ a measure
space, $f_{n}$ $\left(n\in\mathbb{N}\right)$ and $g$ $\mu$-measurable
functions in $E$, $F\subseteq E$ a full measure set such that:
\begin{align*}
 & \forall n\in\mathbb{N}:f_{n}\leq g\quad\mbox{in }F\\
 & \int_{E}g\mbox{d}\mu<+\infty.
\end{align*}
Then
\[
\limsup_{n\to+\infty}\int_{E}f_{n}\mbox{d}\mu\leq\int_{E}\limsup_{n\to+\infty}f_{n}\mbox{d}\mu.
\]
\end{lemma}
\begin{proof}
\textsc{Case I}. $\int_{E}g\mbox{d}\mu=-\infty$. Then 
\[
\limsup_{n\to+\infty}\int_{E}f_{n}\mbox{d}\mu=-\infty
\]
and the thesis is trivially true.

\textsc{Case I}I. $\int_{E}g\mbox{d}\mu\in-\left(\infty,+\infty\right)$

The sequence
\[
a_{n}:=g-\sup_{k\geq n}f_{k}
\]
satisfies
\[
0\leq a_{n}\uparrow g-\limsup_{m\to+\infty}f_{m}\quad\mbox{in }F.
\]
Hence by Monotone convergence:
\begin{equation}
\int_{E}\left(g-\sup_{k\geq n}f_{k}\right)\mbox{d}\mu=\int_{E}a_{n}\mbox{d}\mu\uparrow\int_{E}\left(g-\limsup_{m\to+\infty}f_{m}\right)\mbox{d}\mu.\label{eq: 1}
\end{equation}

Observe that the quantities
\begin{align*}
 & \int_{E}\left(-\sup_{k\geq n}f_{k}\right)\mbox{d}\mu:=\int_{E}\left(g-\sup_{k\geq n}f_{k}\right)\mbox{d}\mu-\int_{E}g\mbox{d}\mu\\
 & \int_{E}\left(-\limsup_{m\to+\infty}f_{m}\right)\mbox{d}\mu:=\int_{E}\left(g-\limsup_{m\to+\infty}f_{m}\right)\mbox{d}\mu-\int_{E}g\mbox{d}\mu
\end{align*}
make sense and belong to $(-\infty,+\infty]$. It follows from $\eqref{eq: 1}$
that:
\begin{equation}
\lim_{n\to+\infty}\int_{E}\left(-\sup_{k\geq n}f_{k}\right)\mbox{d}\mu=\int_{E}\left(-\limsup_{m\to+\infty}f_{m}\right)\mbox{d}\mu.\label{eq: 2}
\end{equation}
Indeed, if $\int_{E}\left(-\sup_{k\geq n_{0}}f_{k}\right)\mbox{d}\mu=+\infty$
for some $n_{0}\in\mathbb{N}$, then both

$\lim_{n\to+\infty}\int_{E}\left(-\sup_{k\geq n}f_{k}\right)\mbox{d}\mu$
and $\int_{E}\left(-\limsup_{m\to+\infty}f_{m}\right)\mbox{d}\mu$
are $+\infty$.

If $\int_{E}\left(-\sup_{k\geq n}f_{k}\right)\mbox{d}\mu<+\infty$
for every $n\in\mathbb{N}$ and

$\int_{E}\left(-\limsup_{m\to+\infty}f_{m}\right)\mbox{d}\mu<+\infty$,
then clearly $\eqref{eq: 2}$ follows from $\eqref{eq: 1}$, whilst
in case

$\int_{E}\left(-\limsup_{m\to+\infty}f_{m}\right)\mbox{d}\mu=+\infty$
we have
\begin{eqnarray*}
+\infty & = & \int_{E}\left(g-\limsup_{m\to+\infty}f_{m}\right)\mbox{d}\mu=\lim_{n\to+\infty}\int_{E}\left(g-\sup_{k\geq n}f_{k}\right)\mbox{d}\mu\\
 & = & \int_{E}g\mbox{d}\mu+\lim_{n\to+\infty}\int_{E}\left(-\sup_{k\geq n}f_{k}\right)\mbox{d}\mu
\end{eqnarray*}
which implies
\[
\lim_{n\to+\infty}\int_{E}\left(-\sup_{k\geq n}f_{k}\right)\mbox{d}\mu=+\infty.
\]

It follows from $\eqref{eq: 2}$ that
\begin{eqnarray*}
\inf_{n\in\mathbb{N}}\int_{E}\sup_{k\geq n}f_{k}\mbox{d}\mu & = & \int_{E}\limsup_{m\to+\infty}f_{m}\mbox{d}\mu.
\end{eqnarray*}
Moreover, it is a consequence of the definition of sup that
\[
\limsup_{m\to+\infty}\int_{E}f_{m}\mbox{d}\mu\leq\inf_{n\in\mathbb{N}}\int_{E}\sup_{k\geq n}f_{k}\mbox{d}\mu.
\]
\end{proof}

\end{document}